\documentclass[12pt,twoside]{amsart}

\usepackage{amssymb}
\usepackage{graphicx}
\usepackage[pdfauthor={David W. Cook II},
            pdftitle={The Lefschetz properties of monomial complete intersections in positive characteristic},
            pdfproducer={LaTeX with hyperref},
            pdfcreator={latex->dvips->ps2pdf},
            pdfborder={0 0 0},
            colorlinks=false]{hyperref}  
\usepackage[margin=1in]{geometry}
\usepackage[all]{xy}

\newcommand{\BP}{\mathcal{B}}
\newcommand{\NN}{\mathbb{N}}
\newcommand{\PP}{\mathbb{P}}
\newcommand{\RR}{\mathbb{R}}
\newcommand{\ZZ}{\mathbb{Z}}
\DeclareMathOperator{\syz}{syz}
\DeclareMathOperator{\Kos}{Kos}
\DeclareMathOperator{\mgd}{mgd}
\newcommand{\st}{\; | \;}                                     
\def\urltilda{\kern -.15em\lower .7ex\hbox{\~{}}\kern .04em}  
\newcommand{\flfr}[2]{\left\lfloor\frac{#1}{#2}\right\rfloor} 
\newcommand{\clfr}[2]{\left\lceil\frac{#1}{#2}\right\rceil}   
\newcommand{\rf}[2]{{\left\langle {#1} \right\rangle}_{#2}}   
\newcommand{\tup}[1]{\underline{#1}}                          
\newcommand{\tupc}[1]{\check{\tup{#1}}}                       
\newcommand{\dotcup}{\ensuremath{\mathaccent\cdot\cup}}       

\numberwithin{figure}{section}
\numberwithin{equation}{section}

\newtheorem{theorem}{Theorem}[section]
\newtheorem{lemma}[theorem]{Lemma}
\newtheorem{proposition}[theorem]{Proposition}
\newtheorem{corollary}[theorem]{Corollary}
\newtheorem{conjecture}[theorem]{Conjecture}

\theoremstyle{definition}

\newtheorem{remark}[theorem]{Remark}

\newtheorem{question}[theorem]{Question}
\newtheorem*{acknowledgement}{Acknowledgement}

\newcount\HOUR
\newcount\MINUTE
\newcount\HOURSINMINUTES
\newcount\INTVAL
\newcommand{\twodigit}[1]{\INTVAL=#1\relax\ifnum\INTVAL<10 0\fi\the\INTVAL}
\HOUR=\time\divide\HOUR by 60\relax
\HOURSINMINUTES=\HOUR\multiply\HOURSINMINUTES by 60\relax
\MINUTE=\time\advance\MINUTE by -\HOURSINMINUTES\relax
\newcommand\rightnow{
           \twodigit{\the\HOUR}:\twodigit{\the\MINUTE},
           \twodigit{\number\day}.\space
           \ifcase\month\or January\or February\or March\or April\or May\or June\or July\or August\or September\or October\or November\or December\fi
           \space\number\year}

\begin{document}

\title[The Lefschetz properties in positive characteristic]{The Lefschetz properties of monomial complete intersections in positive characteristic}
\author[D.\ Cook II]{David Cook II}
\address{Department of Mathematics, University of Kentucky, 715 Patterson Office Tower, Lexington, KY 40506-0027, USA}
\email{\href{mailto:dcook@ms.uky.edu}{dcook@ms.uky.edu}}
\thanks{Part of the work for this paper was done while the author was partially supported by the National Security Agency under Grant Number H98230-09-1-0032.}

\begin{abstract}
    Stanley proved that, in characteristic zero, all artinian monomial complete intersections have the strong Lefschetz property.
    We provide a positive characteristic complement to Stanley's result in the case of artinian monomial complete intersections 
    generated by monomials all of the same degree, and also for arbitrary artinian monomial complete intersections in characteristic two.

    To establish these results, we first prove an {\em a priori} lower bound on the characteristics that guarantee the Lefschetz properties.
    We then use a variety of techniques to complete the classifications.
\end{abstract}

\maketitle

\section{Introduction}\label{sec:intro}

Let $K$ be an infinite field of arbitrary characteristic, and let $I$ be a homogeneous artinian ideal in $R = K[x_0, \ldots, x_n]$.
The algebra $A = R/I$ is said to have the {\em strong Lefschetz property}, if there exists a linear form $\ell \in A$ such that for
all integers $d$ and $k$, with $k \geq 1$, the map $\times \ell^k : [A]_d \rightarrow [A]_{d+k}$ has maximal rank.  In this case, $\ell$
is called a {\em strong Lefschetz element} of $A$.  If the property holds for $k = 1$, then $A$ is said to have the 
{\em weak Lefschetz property}, and $\ell$ is called a {\em weak Lefschetz element} of $A$.

The Lefschetz properties have been studied extensively; the recent manuscript by Harima, Maeno, Morita, Numata, Wachi, and 
Watanabe~\cite{HMMNWW} provides a wonderfully comprehensive exploration of the Lefschetz properties.  In particular, the presence
of the properties provides interesting constraints on the Hilbert functions of the algebras (see, e.g., \cite{BMMNZ, HMNW, MZ}).

The theorem below which motivates our results was first proven by Stanley~\cite{St} using algebraic topology.  It has
since been proven in many different ways, most notably by Watanabe~\cite{Wa} using representations of $\mathfrak{gl}_2$, and later
by Reid, Roberts, and Roitman~\cite{RRR} using purely algebraic techniques.

\begin{theorem}{\rm (\cite[Theorem~2.4]{St}, \cite[Corollary~3.5]{Wa}, \cite[Theorem~10]{RRR})} \label{thm:char-zero}
    Every artinian monomial complete intersection over a polynomial ring has the strong Lefschetz property in characteristic zero.
\end{theorem}

For a brief, but extensive overview of the depth with which Theorem~\ref{thm:char-zero} has inspired explorations of the Lefschetz properties, 
see the survey~\cite{MN} by Migliore and Nagel.

We emphasise that the above result, and most related results, are specific to characteristic zero.  However, there has been a great deal
of recent interest in positive characteristic (see, e.g., \cite{CGJL, CN-Enum, LZ}).  Specifically, Brenner and Kaid~\cite{BK-p} 
(for three variables) and Kustin and Vraciu~\cite{KV} (for at least four variables) completely characterised the characteristics
in which the weak Lefschetz property is present for monomial complete intersections generated by monomials all having the same degree.  Although the failure
of the weak Lefschetz property implies the failure of the strong Lefschetz property, we must do more work to establish the presence of the
strong Lefschetz property.

The goal of this note is to provide complements to Theorem~\ref{thm:char-zero} in characteristic two (see Theorem~\ref{thm:classify-char-two})
and further in the case of generation by monomials of the same degree (see Theorem~\ref{thm:classify-fixed-d}).  The remainder of the manuscript
is organised as follows:  In Section~\ref{sec:est} we describe a few old and new ways to establish the Lefschetz properties, specifically 
in the case of monomial complete intersections.  In Section~\ref{sec:bound} we describe the characteristics in which the Lefschetz properties
may fail, and prove they are bounded linearly in the degrees of the generating monomials.  The proofs involve an analysis of the prime divisors
of an associated determinant.

As demonstrated in \cite{KV}, when fewer variables are used, exploring the presence of the Lefschetz properties becomes more interesting.  
In Sections~\ref{sec:two} and~\ref{sec:three} we consider monomial complete intersections in two and three variables, respectively.  In 
Section~\ref{sec:many} we handle the case of at least four variables.  Throughout these three sections, we use a variety of techniques
to establish the presence and failure of the Lefschetz properties.  These techniques include determining syzygy gaps (see Subsection~\ref{sub:syzygy}), 
using basic number theory (e.g., see Lemma~\ref{lem:one-or-the-other-n-1}), and finding explicit syzygies of small degree (see Subsection~\ref{sub:mgd}).

Finally, in Section~\ref{sec:conclusions} we close with the desired classifications and a few comments.

\section{Establishing the Lefschetz properties}\label{sec:est}

Let $K$ be an infinite field of arbitrary characteristic.  All artinian monomial complete intersections over the polynomial ring 
$R =  K[x_0,\ldots, x_n]$ are of the form $R/I_{\tup{d}}$, where
\[
    I_{\tup{d}} = (x_0^{d_0}, x_1^{d_1}, \ldots, x_n^{d_n}),
\]
$\tup{d} = (d_0, d_1, \ldots, d_n) \in \NN^{n+1}$, and, without loss of generality, $d_0 \geq d_1 \geq \cdots \geq d_n \geq 2$.
Throughout the remainder of the manuscript we use the above definition of $I_{\tup{d}}$.

\subsection{The weak Lefschetz property}

Notice that the socle degree of $R/I_{\tup{d}}$ is $t := d_0 + \ldots + d_n - (n+1)$.  Moreover, if the largest generating
degree is sufficiently large (relative to the socle degree), then the weak Lefschetz property always holds.

\begin{proposition}{\cite[Proposition~5.2]{MM}} \label{pro:large-degree}
    Let $\tup{d} \in \NN^{n+1}$, $d_0 \geq d_1 \geq \cdots \geq d_n \geq 2$, and $t = d_0 + \cdots + d_n - (n+1)$.
    If $d_0 > \clfr{t}{2}$, then $R/I_{\tup{d}}$ has the weak Lefschetz property, regardless of the characteristic of $K$.
\end{proposition}

We recall that for level algebras if multiplication by a general linear form is injective in a fixed degree, then
it is injective for all the previous degrees.

\begin{proposition}{\cite[Proposition~2.1(b-c)]{MMN}} \label{pro:mmn21}
    Let $A = R/I$ be a level artinian standard graded $K$-algebra, and let $\ell$ be a general linear form.  Consider the
    homomorphisms $\varphi_d: [A]_d \rightarrow [A]_{d+1}$ defined by multiplication by $\ell$ for $d \geq 0$. 
    \begin{enumerate}
        \item If $\varphi_d$ is injective for some $d \geq 1$, then $\varphi_{d-1}$ is injective.
        \item In particular, if $\dim_K [A]_{d} = \dim_K [A]_{d+1}$, then $A$ has the
            weak Lefschetz property if and only if $\varphi_d$ is injective (or, equivalently, surjective).
    \end{enumerate}
\end{proposition}

Moreover, a monomial algebra has the weak (strong) Lefschetz property exactly when the sum of the variables is a weak (strong) Lefschetz element.

\begin{proposition}{\cite[Proposition~2.2]{MMN}} \label{pro:mono}
    Let $A = R/I$ be an artinian standard graded $K$-algebra with $I$ generated by monomials.  Then $A$ has the weak (strong)
    Lefschetz property if and only if $x_0 + \cdots + x_n$ is a weak (strong) Lefschetz element of $A$.
\end{proposition}

Thus, in the case of an artinian monomial complete intersection, we have a series of conditions on the algebra that are
equivalent to the algebra having the weak Lefschetz property.

\begin{lemma} \label{lem:wlp-odd}
    Let $\ell = x_0 + \cdots + x_n$.  Suppose $t$ is odd and set $s = \flfr{t}{2}$.  Then the following are equivalent (where
    the ordering on the $d_i$ is ignored):
    \begin{enumerate}
        \item The algebra $R/I_{\tup{d}}$ has the weak Lefschetz property;
        \item the multiplication map $\times \ell: [R/I_{\tup{d}}]_{s} \rightarrow [R/I_{\tup{d}}]_{s+1}$ is an injection;
        \item the $K$-dimension of $[R/(I_{\tup{d}}, \ell)]_{s+1}$ is $0$;
        \item the $K$-dimension of $[S/J_{\tup{d}}]_{s+1}$ is $0$, where $S = K[x_1,\ldots,x_n]$ and
            \[
                J_{\tup{d}} = ((x_1+\cdots+x_n)^{d_0}, x_1^{d_1}, \ldots, x_n^{d_n}).
            \]
    \end{enumerate}
\end{lemma}
\begin{proof}
    By Proposition~\ref{pro:mono}, as $I_{\tup{d}}$ is a monomial ideal, it suffices to consider $\ell = x_0 + \ldots + x_n$.

    The equivalences follow as:
    \begin{itemize} 
        \item[] (i) \& (ii): use Proposition~\ref{pro:mmn21}(ii) and duality;
        \item[] (ii) \& (iii): $[R/(I_{\tup{d}}, \ell)]_{s+1}$ is the cokernel of the map in (ii); and
        \item[] (iii) \& (iv): $[R/(I_{\tup{d}}, \ell)]_{s+1} \cong [S/J_{\tup{d}}]_{s+1}$.\qedhere
    \end{itemize}
\end{proof}

If the socle degree is even, then the weak Lefschetz property is sometimes inherited.

\begin{corollary} \label{cor:wlp-even}
    If $t$ is even and $R[x_{n+1}]/I_{(\tup{d},2)}$ has the weak Lefschetz property, then $R/I_{\tup{d}}$ has the weak Lefschetz property.
\end{corollary}
\begin{proof}
    Notice that $s = \frac{t}{2} = \flfr{t+1}{2}$ as $t$ is even.  Set $\ell := x_0 + \cdots x_n$.

    By Lemma~\ref{lem:wlp-odd}, if $R[x_{n+1}]/I_{(\tup{d},2)}$ has the weak Lefschetz property, then $K$-dimension of 
    $[R/(I_{\tup{d}}, \ell^2)]_{s+1}$ is zero.  Moreover, this is the cokernel of the map $\times \ell^2: [R/I_{\tup{d}}]_{s-1} \rightarrow [R/I_{\tup{d}}]_{s+1}$;
    hence the map is a bijection.  This implies $\times \ell: [R/I_{\tup{d}}]_{s-1} \rightarrow [R/I_{\tup{d}}]_{s}$ is an injection.
    Thus, using Proposition~\ref{pro:mmn21}(ii) and duality we have that $R/I_{\tup{d}}$ has the weak Lefschetz property.
\end{proof}

\subsection{The strong Lefschetz property}

If the Hilbert function is symmetric, then demonstrating the strong Lefschetz property is equivalent to showing certain maps are bijections.

\begin{remark} \label{rem:slp-is-ssp}
    An artinian algebra $A$ with socle degree $t$ has the {\em strong Stanley property} if there exists a linear form $\ell \in A$ such that the 
    map $\times \ell^{t-2k}: [A]_k \rightarrow A_{t-k}$ is a bijection for all $0 \leq k \leq \flfr{t}{2}$.  Clearly then, an artinian algebra has the 
    strong Stanley property if and only if the algebra has the strong Lefschetz property and has a symmetric Hilbert function.  
\end{remark}

It follows that $R/I_{\tup{d}}$ has the strong Stanley property if and only if it has the strong Lefschetz property.  Moreover, this
provides a deeper connection between the strong and weak Lefschetz properties.

\begin{proposition} \label{pro:slp}
    $R/I_{\tup{d}}$ has the strong Lefschetz property if and only if $R[x_{n+1}]/I_{(\tup{d}, t-2k)}$ has the weak Lefschetz property
    for each $0 \leq k \leq \flfr{t}{2}$.
\end{proposition}
\begin{proof}
    Set $\ell := x_0 + \cdots x_n$.  The socle degree of $A_k := R[x_{n+1}]/I_{(\tup{d}, t-2k)}$ is $2(t-k)-1$, which is odd, and further 
    $\flfr{2(t-k)-1}{2} = t-k-1$.  By Lemma~\ref{lem:wlp-odd}, $A_k$ has the weak Lefschetz property if and only if the $K$-dimension of 
    $[R/(I_{\tup{d}}, \ell^{t-2k})]_{t-k}$ is zero.  The latter is equivalent to the map 
    $\varphi_k := \times \ell^{t-2k}: [R/I_{\tup{d}}]_{k} \rightarrow [R/I_{\tup{d}}]_{t-k}$ being a bijection.

    By Remark~\ref{rem:slp-is-ssp}, $R/I_{\tup{d}}$ has the strong Lefschetz property if and only if it is the strong Stanley property, that is,
    if and only if $\varphi_k$ is a bijection (i.e., $A_k$ has the weak Lefschetz property) for $0 \leq k \leq \flfr{t}{2}$.
\end{proof}

\section{Bounding failure of the Lefschetz properties}\label{sec:bound}

Let $R/I_{\tup{d}}$ be an artinian monomial complete intersection as defined in Section~\ref{sec:est}, and let $t = d_0 + \cdots + d_n - (n+1)$ be
the socle degree of $R/I_{\tup{d}}$.  If $t$ is odd and $d_0 \leq \clfr{t}{2}$, then Lemma~\ref{lem:wlp-odd}(iv) holds if and only if the determinant
of $M_{\tup{d}}$, the associated matrix defined by the map, is non-zero modulo the characteristic of $K$.  We use this to 
describe the characteristics in which the Lefschetz properties may fail and to prove they are bounded linearly in the degrees of the generating monomials.

\subsection{A connection to weak compositions}

An ordered $n$-tuple $\tup{m} = (m_1, \ldots, m_n) \in \NN_0^{n}$ with $m_1 + \cdots + m_n = k$ is called a {\em weak composition} of $k$ into
$n$ parts.  Define the set $C(n, \tup{m}, k)$ to be the set of weak compositions $\tup{a}$ of $k$ into $n$ parts such that $\tup{a}$ is
component-wise bounded by $\tup{m}$.  For elements $\tup{a}, \tup{b} \in C(n, \tup{m}, k)$, define $\tup{a}! := a_1!\cdots a_n!$ and 
$\tup{b} - \tup{a} = (b_1 - a_1, \ldots, b_n - a_n)$.  Notice that if $\tup{a}$ is component-wise bounded by $\tup{b}$, then 
$\tup{b} - \tup{a} \in C(n, \tup{m}, k)$.  

Given an $n$-tuple $\tup{a} = (a_1, a_2, \ldots, a_n)$, we define $x^{\tup{a}} = x_1^{a_1} \cdots x_n^{a_n}$.
The matrix $M_{\tup{d}}$ has rows indexed by the monomials $x^{\tup{a}}$ of $[S/J_{\tup{d}}]_{s+1-d_0}$ and columns indexed by the monomials
$x^{\tup{b}}$ of $[S/J_{\tup{d}}]_{s+1}$.  The element in the $x^{\tup{a}}$ row and the $x^{\tup{b}}$ column is zero if $\tup{a}$ is larger than $\tup{b}$ 
in at least one component, otherwise it is the multinomial coefficient
\[
    \binom{d_0}{b_1 - a_1, \ldots, b_n - a_n} = \frac{d_0!}{(\tup{b} - \tup{a})!}.
\]

Notice that the monomials in $[S/J_{\tup{d}}]_i$ are in bijection with the weak compositions in $C(n, \tupc{d} - \tup{1}, i)$, where
$\tupc{d} = (d_1, \ldots, d_n)$ and $\tup{1} = (1, \ldots, 1)$.  Hence the matrix $M_{\tup{d}}$ can be seen as a matrix with rows indexed
by $\tup{a} \in C(n, \tupc{d} - \tup{1}, s+1-d_0)$ and columns indexed by $\tup{b} \in C(n, \tupc{d} - \tup{1}, s+1)$ with entries
given by zero if $\tup{a}$ is larger than $\tup{b}$ in at least one component and $\frac{d_0!}{(\tup{b} - \tup{a})!}$ otherwise.

Seeing $M_{\tup{d}}$ in this new light, a theorem of Proctor computes the determinant of $M_{\tup{d}}$ in terms of compositions.

\begin{theorem}{\cite[Corollary~1]{Pr}} \label{thm:det}
    Let $\tup{d} \in \NN^{n+1}$, where $d_0 \geq d_1 \geq \cdots \geq d_n \geq 2$, and suppose $d_0+\cdots+d_n-(n+1)$ is odd.
    Set $s := \flfr{d_0+d_1+\cdots+d_n-(n+1)}{2}$.  Then
    \[
        |\det{M_{\tup{d}}}| = 
            \frac{\prod_{\tup{a}} \tup{a}!}{\prod_{\tup{b}} \tup{b}!}
            \prod_{i = 0}^{s+1-d_0} \rf{i+1}{d_0}^{\delta_{s+1-d_0-h}},
    \]
    where $\tup{a}$ and $\tup{b}$ run over $C(n, \tupc{d} - \tup{1}, s+1-d_0)$ and $C(n, \tupc{d} - \tup{1}, s+1)$, respectively,
    $\rf{x}{m} := x(x+1)\cdots(x+m-1)$, and $\delta_i = \#C(n, \tupc{d} - \tup{1}, i) - \#C(n, \tupc{d} - \tup{1}, i-1)$.
\end{theorem}

\begin{remark} \label{rem:nilps}
    By the work of Gessel and Viennot \cite{GV}, we have that the determinant of $M_{\tup{d}}$ is the enumeration
    of signed non-intersecting lattice paths from the hyperplane $x_1 + \cdots + x_n = s+1 - d_0$ to the hyperplane
    $x_1 + \cdots + x_n = s+1$ in the parallelepiped of size $(d_1-1) \times \cdots \times (d_n-1)$ .
\end{remark}

If the top generating degree, $d_0$, is as large as possible such that the preceding theorem is still applicable, then the matrix has one entry.

\begin{lemma} \label{lem:large-top}
    Let $n \geq 2$ and $d_1 \geq \cdots \geq d_n \geq 2$; set $d_0 = d_1 + \cdots + d_n - n$.  Then the algebra 
    $R/I_{(d_0, d_1, \ldots, d_n)}$ has the weak Lefschetz property if and only if $p$ does not divide 
    $\binom{d_0}{d_1-1, \ldots, d_n-1}$.
\end{lemma}
\begin{proof}
    The socle degree is $t = 2(d_1 + \cdots + d_n - n) - 1 = 2d_0 - 1$, and so the peak is $s = d_0$.  Thus, $M_{\tup{d}}$ is the
    $1 \times 1$ matrix with entry $\binom{d_0}{d_1-1, \ldots, d_n-1}$, and so $\det{M_{\tup{d}}} = \binom{d_0}{d_1-1, \ldots, d_n-1}$.
\end{proof}

\subsection{Bounding failure}

Using the above connection, and some algebraic considerations, we can bound the prime characteristics in which the weak Lefschetz property
can fail.  We first recall a useful lemma.

\begin{lemma}{\cite[Lemma~2.5]{CN-Enum}} \label{lem:wlp-p}
    Let $A = R/I$ be an artinian standard graded $K$-algebra with $I$ generated by monomials.  Suppose that $a$ is the least positive integer
    such that $x_i^a \in I$, for $1 \leq i \leq n$, and suppose that the Hilbert function of $R/I$ weakly increases to degree $s$.
    Then, for any positive prime $p$ such that $a \leq p^m \leq s$ for some positive integer $m$, $A$ fails to have the weak Lefschetz property
    in characteristic $p$.
\end{lemma}

\begin{proposition} \label{pro:wlp-naive}
    Let $n \geq 2$ and $d_0 \geq \cdots \geq d_n \geq 2$; set $t = d_0 + \cdots + d_n - n$.  Suppose $K$ is a field of characteristic $p$,
    where $p$ is a positive prime, and suppose $d_0 \leq \clfr{t}{2}$.  Then:
    \begin{enumerate}
        \item If $d_1 \leq p \leq d_0$ or $d_0 \leq p^m \leq \clfr{t}{2}$, for some positive integer $m$,
            then $R/I_{\tup{d}}$ fails to have the weak Lefschetz property.  In particular, injectivity fails in degree $d_0$ or $p^m$, respectively.
        \item If $p > \clfr{t+1}{2}$, then $R/I_{\tup{d}}$ has the weak Lefschetz property.
    \end{enumerate}
\end{proposition}
\begin{proof}
    Set $\ell := x_0 + \cdots + x_n$ and $\varphi_k := \times \ell: [R/I_{\tup{d}}]_{k-1} \rightarrow [R/I_{\tup{d}}]_{k}$.

    Assume $d_1 \leq p \leq d_0$.  Then $\ell^{d_0}$ is zero in $R/I_{\tup{d}}$ as the coefficients $\binom{d_0}{i_1, \ldots, i_n}$ are
    zero modulo $p$ except for on $x_i^{d_0}$, $0 \leq i \leq n$, but these are in $I_{\tup{d}}$.  Hence $\varphi_{d_0}$ is not injective
    and so $R/I_{\tup{d}}$ fails to have the weak Lefschetz property.

    Next, assume $d_0 \leq p^m \leq \clfr{t}{2}$, for some positive integer $m$.  Then $\varphi_{p^m}$ is not injective and $R/I_{\tup{d}}$ fails to
    have the weak Lefschetz property by Lemma~\ref{lem:wlp-p}.  (Recall that $x_i^{d_0} \in I_{\tup{d}}$ for $0 \leq i \leq n$ and the Hilbert
    function of $R/I_{\tup{d}}$ weakly increases to $t - \flfr{t}{2} = \clfr{t}{2}$.)

    Finally, assume $p > \clfr{t+1}{2}$.  We consider the two cases given by the parity of $t$.

    Suppose $t$ is odd; then $\clfr{t+1}{2} = \clfr{t}{2}$.  Moreover, analysing Theorem~\ref{thm:det} we see that the terms in the formula are bounded 
    between $1$ and $\clfr{t+1}{2}$.  Thus, $\det{M_{\tup{d}}}$ is not divisible by primes $p > \clfr{t+1}{2}$, and $R/I_{\tup{d}}$ has the weak Lefschetz
    property if $p > \clfr{t+1}{2}$.  

    Suppose $t$ is even; then $\clfr{t+1}{2} = \frac{t}{2}+1 = \clfr{t+2}{2}$.  By the previous paragraph, $R[x_{n+1}]/I_{\tup{d},2}$ has the weak Lefschetz 
    property for $p > \clfr{t+2}{2} = \clfr{t+1}{2}$.  Hence, by Corollary~\ref{cor:wlp-even} $R/I_{\tup{d}}$ has the weak Lefschetz property
    if $p > \clfr{t+1}{2}$.
\end{proof}

Notice that the algebras in Proposition~\ref{pro:slp}, which we desire to show have the weak Lefschetz property, all have odd socle degree.  We
exploit this, along with the preceding proposition, to find a similar bound in the case of the strong Lefschetz property.

\begin{theorem} \label{thm:slp}
    Suppose $K$ is a field of characteristic $p$, where $p$ is a positive prime.  Then:
    \begin{enumerate}
        \item If $\max\{d_1, 2d_0 - t\} \leq p \leq d_0$ or $d_0 \leq p^m \leq t$, for some positive integer $m$,
            then $R/I_{\tup{d}}$ fails to have the strong Lefschetz property.
        \item If $p > t$, then $R/I_{\tup{d}}$ has the strong Lefschetz property.
    \end{enumerate}
\end{theorem}
\begin{proof}
    Set $\ell := x_0 + \cdots + x_n$ and $A_k := R[x_{n+1}]/I_{(\tup{d}, t-2k)}$.  Recall that by Proposition~\ref{pro:slp},
    $R/I_{\tup{d}}$ has the strong Lefschetz property if and only if each $A_k$, for $0 \leq k \leq \flfr{t}{2}$, has the
    weak Lefschetz property.  Set $r := \min\{\flfr{t}{2}, t-d_0\}$ and notice that the largest generating degree of $A_k$ is $\max\{d_0, t-2k\}$.
    Thus $A_k$ satisfies the hypotheses of Proposition~\ref{pro:wlp-naive} if and only if $0 \leq k \leq r$.

    Suppose $d_0 \leq \clfr{t}{2}$ and $\max\{d_1, 2d_0 - t\} \leq p \leq d_0$ or $d_0 \leq p^m \leq t$, for some positive integer $m$.
    Then $\max\{d_1, 2d_0 - t\} = d_1$, and by Proposition~\ref{pro:wlp-naive}(i) $R/I_{\tup{d}}$ fails to have the weak Lefschetz property,
    hence fails to have the strong Lefschetz property.

    Suppose $d_0 > \clfr{t}{2}$ and $\max\{d_1, 2d_0 - t\} \leq p \leq d_0$.  We then have that $0 < t - d_0 < \flfr{t}{2}$ and $A_{t-d_0}$ 
    fails to have the weak Lefschetz property by Proposition~\ref{pro:wlp-naive}(i).

    Let $0 \leq k \leq r$.  Then by Proposition~\ref{pro:wlp-naive}(i), $A_k$ fails to have the weak Lefschetz property if
    $\max\{t-2k, d_0\} \leq p^m \leq t-k$, for some positive integer $m$.  Hence ranging $k$ from $0$ to $r$ we get that $R/I_{\tup{d}}$
    fails to have the weak Lefschetz property for $d_0 \leq p^m \leq t$.

    On the other hand, by Proposition~\ref{pro:wlp-naive}(ii), $A_k$ has the weak Lefschetz property for $p > t-k$.  Hence if $p > t$, then
    each $A_k$ has the weak Lefschetz property and $R/I_{\tup{d}}$ has the strong Lefschetz property.
\end{proof}

Case (ii) of the preceding theorem can be recovered with some work from results of Lindsey \cite[Lemma~5.2 and Corollary~5.3]{Lind},
or Hara and Watanabe's proof of~\cite[Proposition 8]{HW}.

\section{The presence of the Lefschetz properties for two variables}\label{sec:two}

First, we note that any homogeneous artinian ideal in two variables has the weak Lefschetz property.  This was proven for characteristic zero 
in~\cite[Proposition~4.4]{HMNW} and then for arbitrary characteristic in~\cite[Corollary~7]{MZ}, though it was not specifically stated therein, as
noted in~\cite[Remark~2.6]{LZ}.  (See also~\cite[Proposition~2.7]{CN-Enum}.)

\begin{proposition} \label{pro:2-wlp}
    Let $R = K[x,y]$, where $K$ is an infinite field with {\em arbitrary} characteristic.  Every homogeneous artinian algebra in $R$ has the weak Lefschetz property.
\end{proposition}

On the other hand, the strong Lefschetz property is much more subtle.  By Proposition~\ref{pro:slp}, $R/I_{(a,b)}$ has the strong Lefschetz property if and only if
$B_k = R[x_2]/I_{(a,b,a+b-2-2k)}$ has the weak Lefschetz property for $0 \leq k \leq \flfr{a+b-2}{2}$.  In this case, if $k > b - 2$, then $B_k$ always has the weak 
Lefschetz property, by Proposition~\ref{pro:large-degree}; hence we need only to consider $0 \leq k \leq b-2$.

A few particular cases stand out.  Using Lemma~\ref{lem:large-top}, we have that the algebra $B_0$ has the weak Lefschetz property in characteristic $p$ if and only if $p$ 
does not divide $\binom{a+b-2}{b-1}$.  Similarly, $B_{b-2}$ has the weak Lefschetz property in characteristic $p$ if and only if $p$ does not divide $\binom{a}{b-1}$.

For $2 \leq b \leq 3$, we characterise the strong Lefschetz property with the above.  We single out these cases because they play a special role
in the classification of the strong Lefschetz property in characteristic two for arbitrary $R/I_{\tup{d}}$ given in Section~\ref{sec:conclusions}.
\begin{lemma} \label{lem:2-slp-small-b}
    Let $R = K[x,y]$ and $p$ be the characteristic of $K$.  Then:
    \begin{enumerate}
        \item $R/I_{(a,2)}$, for $a \geq 2$, has the strong Lefschetz property if and only if $p$ does not divide $a$.
        \item $R/I_{(a,3)}$, for $a \geq 3$, has the strong Lefschetz property if and only if $p = 2$ and $a\equiv 2 \pmod{4}$ or 
            $p \neq 2$ and $a$ is not equivalent to $-1$, $0$, or $1$ modulo $p$
    \end{enumerate}
\end{lemma}
\begin{proof}
    By the comments above, $R/I_{(a,2)}$ has the strong Lefschetz property if and only if $p$ does not divide $\binom{a}{1} = a$.

    Similarly, $R/I_{(a,3)}$ has the strong Lefschetz property if and only if $p$ does not divide either $\binom{a+1}{2}$ or $\binom{a}{2}$.  That is,
    $p$ does not divide $\binom{a+1}{2}\binom{a}{2} = \frac{1}{4}(a-1)a^2(a+1)$.  Analysing this, we see that this is equivalent to the claim.
\end{proof}

\subsection{Syzygy gaps} \label{sub:syzygy}

Let $a \geq b$, and let $B_k = R[x_2]/I_{(a,b,a+b-2-2k)}$ for $0 \leq k \leq b-2$.  Notice that $a+b > a+b-2-2k$ as $k \geq 0$, and $b + (a+b-2-2k) > a$ as $b-2 \geq k$.
Thus, by \cite[Corollary~3.2]{Br}, $B_k$ has a stable syzygy bundle, and so by \cite[Theorem~2.2]{BK}, $B_k$ has the weak Lefschetz property if and only if the syzygy
bundle of $B_k$ splits on the line $x+y+z$ with twists, say, $s_0 \leq s_1$, such that $s_1 - s_0 \leq 1$, i.e., the {\em syzygy gap} (introduced by Monsky \cite{Mon}) of 
$(x^a, y^b, (x+y)^{a+b-2-2k})$ in $R$ is at most one.  Moreover, it is easy to see that $s_0 + s_1 = -2(a+b-1-k)$, and hence the parity of the syzygy gap $s_1 - s_0$ is even.

Han~\cite{Han} provides a way to compute the syzygy gap via a continuation of the syzygy gap function.  Define $\delta: \NN^3 \rightarrow \NN_0$ to be the syzygy gap of 
$(x^a, y^b, (x+y)^c)$ in $K[x,y]$ (notice $\delta$ depends on the characteristic of $K$).  Let $\delta^\star: [0,\infty)^3 \rightarrow [0, \infty)$ be the continuous
continuation of $\delta$.  Define $\ZZ^3_{\rm odd}$ to be the integer triples $(u,v,w)$ such that $u+v+w$ is odd.  Further, define $\mu$, the {\em Manhattan distance} 
on $\RR^3$, to be $\mu((a,b,c),(u,v,w)) = |u-a| + |v-b| + |w-c|$.

\begin{theorem}{\cite[Theorems~2.25 and 2.29]{Han}} \label{thm:Han}
    Let $K$ be an algebraically closed field of characteristic $p > 0$, and assume the entries $(a,b,c) \in [0,\infty)^3$ satisfy $a\leq b \leq c < a+b$.
    If there exists a negative integer $s$ and a triple $(u,v,w) \in \ZZ^3_{\rm odd}$ such that $\mu(p^s(a,b,c), (u,v,w)) < 1$, then $\delta^\star(a,b,c) > 0$.
    Otherwise, if no such $s$ and $(u,v,w)$ exist, then $\delta^\star(a,b,c) = 0$.
\end{theorem}

This approach was used by Brenner and Kaid in~\cite{BK-p} to classify the characteristics in which $K[x,y,z]/(x^d, y^d, z^d)$ has the weak Lefschetz property.

\subsection{Characteristic two}

Let $n$ be a positive integer, and $n = b_s 2^s + \cdots + b_0 2^0$ be its binary representation, i.e., $b_i \in \{0, 1\}$.  Define the {\em bit-positions}
of $n$ to be the set $\BP(n)$ of indices $i$ such that $b_i = 1$ in the binary representation of $n$.  For example, $\BP(42) = \{1,3,5\}$ and $\BP(2^m) = \{m\}$,
for $m \geq 0$.

The following theorem is due to Kummer~\cite{Ku}.  (We thank Fabrizio Zanello for pointing us to this reference.)

\begin{theorem} \label{thm:Kummer}
    If $n \geq k \geq 0$ and $p$ is a prime, then the largest power of $p$ dividing $\binom{n}{k}$ is
    the number carries that occur in the addition of $k$ and $n-k$ in base-$p$ arithmetic.
\end{theorem}

An immediate (and simple) corollary of this theorem is a classification for when binomial coefficients are odd.

\begin{corollary} \label{cor:Kummer}
    If $n \geq k \geq 0$, then $\binom{n}{k}$ is odd if and only if $\BP(k)$ and $\BP(n-k)$ are disjoint.
\end{corollary}

Using the above classification, we get a useful intermediate result.

\begin{lemma} \label{lem:one-or-the-other-n-1}
    If $a \geq b \geq 2$, then $\binom{a+b-2}{b-1}$ is odd and $\binom{a}{b-1}$ is odd if and only if $a = 2^m \ell$ and $b = 2^m+1$, where $m \geq 0$ and $l \geq 3$ odd.
\end{lemma}
\begin{proof}
    Suppose $\binom{a+b-2}{b-1}$ and $\binom{a}{b-1}$ are odd.  Then by Corollary~\ref{cor:Kummer}, $\BP(a-1)$ and $\BP(b-1)$ are disjoint and as are $\BP(b-1)$ and $\BP(a-b+1)$.
    Thus, $\BP(a) = \BP(b-1) \dotcup \BP(a-b+1)$ has at least two elements, as $\BP(b-1)$ and $\BP(a-b+1)$ each have at least one element.

    Suppose $\BP(a-1)$ contains $0, \ldots, m-1$ but not $m$.  Then $\BP(a)$ contains $m$ but not $0, \ldots, m-1$.  Further, for $i > m$, $i \in \BP(a-1)$ if and only
    if $i \in \BP(a)$.  As $\BP(b-1) \subset \BP(a)$, and $\BP(b-1)$ has at least one element, then $\BP(b-1)$ and $\BP(a-1)$ being disjoint implies $\BP(b-1) = \{m\}$.  
    That is, $b - 1 = 2^m$ and so $b = 2^m+1$.  Moreover, as $\BP(a)$ contains $m$ but not $0, \ldots, m-1$, $a = 2^m \ell$, where $\ell$ is odd.  As $a \geq b$,
    then $a \geq 3$.  

    On the other hand, suppose $a = 2^m \ell$ and $b = 2^m +1$.  Then $\BP(a-1) = \{0, \ldots, m-1\} \dotcup \{ m+1+i \st i \in B(\ell)\}$, $B(b-1) = \{m\}$,
    and $\BP(a-b+1) = \BP(2^{m+1}\frac{\ell-1}{2})$ only contains values at least $m+1$.  Thus $\BP(a-1)$ and $\BP(b-1)$ are disjoint as are $\BP(b-1)$ and $\BP(a-b+1)$,
    and so by Corollary~\ref{cor:Kummer} we have that $\binom{a+b-2}{b-1}$ and $\binom{a}{b-1}$ are odd.
\end{proof}

Using the syzygy gap method described in Subsection~\ref{sub:syzygy}, we complete the classification in the exceptional cases.

\begin{lemma} \label{lem:n-1-special-wlp}
    If $a = 2^m \ell$ and $b = 2^m+1$, where $m \geq 0$ and $l \geq 3$ odd, then $R/I_{(a,b,a+b-2(k+1))}$ fails to have the weak Lefschetz property
    in characteristic two if and only if $1 \leq k \leq b-3$.
\end{lemma}
\begin{proof}
    If $k=0$ or $k = b-2$, then $R/I_{(a,b,a+b-2)}$ or $R/I_{(a,b,a-b+2)}$ has the weak Lefschetz property by Lemmas~\ref{lem:large-top} and~\ref{lem:one-or-the-other-n-1}.

    Suppose $1 \leq k \leq b-3 = 2^m-2$.  We have the following:
    \begin{equation*}
        \begin{split}
            \mu\left( \frac{1}{2^m}(a,b,a+b-2-2k), (\ell, 1, \ell) \right) 
                            =& \left| \frac{2^m \ell}{2^m} - \ell \right| + \left| \frac{2^m + 1}{2^m} - 1 \right| \\
                             &+ \left| \frac{2^m(\ell+1)-1-2k}{2^m} - \ell \right| \\
                            =& \frac{1}{2^m} + \left| \frac{2k+1}{2^m} - 1 \right| \\
                            =& \left\{\begin{array}{ll}
                                    1-\frac{k}{2^{m-1}} & {\rm if } k \leq 2^{m-1}, \\
                                    \frac{k+1}{2^{m-1}} - 1 & {\rm if } k > 2^{m-1}. \\
                               \end{array}\right.
        \end{split}
    \end{equation*}
    Notice that $1 - \frac{k}{2^{m-1}} < 1$ if and only if $\frac{k}{2^{m-1}} > 0$ if and only if $k > 0$.  Further,
    $\frac{k+1}{2^{m-1}} - 1 < 1$ if and only if $k < 2^m-1$ if and only if $k \leq 2^m-2 = b-3$. 

    Thus, $\mu\left( \frac{1}{2^m}(a,b,a+b-2-2k), (\ell, 1, \ell) \right) < 1$ for $1 \leq k \leq b-3$.  Notice $2\ell + 1$ is odd.
    Hence, by Theorem~\ref{thm:Han}, $R/I_{(a,b,a+b-2(1+k))}$ fails to have the weak Lefschetz property in characteristic two.
\end{proof}

Combining the above two lemmas, we classify the strong Lefschetz property in characteristic two for the two-variable case.

\begin{corollary} \label{cor:char-2-n-1}
    Let $a \geq b \geq 2$.  Then $R/I_{(a,b)}$ fails to have the strong Lefschetz property in characteristic two if and only if one of the
    following hold:
    \begin{enumerate}
        \item $b = 2$ and $a$ is even,
        \item $b = 3$ and $a \not\equiv 2 \pmod{4}$, or
        \item $b \geq 4$.
    \end{enumerate}
\end{corollary}
\begin{proof}
    Parts (i) and (ii) follow from Lemma~\ref{lem:2-slp-small-b} (and also Lemma~\ref{lem:one-or-the-other-n-1} and Lemma~\ref{lem:n-1-special-wlp},
    after considering each case).

    Recall that by Proposition~\ref{pro:slp}, $R/I_{(a,b)}$ has the strong Lefschetz property if and only if each $B_k := S/I_{(a,b,a+b-2-2k)}$ has the weak Lefschetz 
    property, for $0 \leq k \leq b-2$.

    Suppose that $b \geq 4$.  If $a \neq 2^m \ell$ for some $m \geq 0$ or $b \neq 2^m+1$ for some $l \geq 3$ odd, then by Lemma~\ref{lem:one-or-the-other-n-1},
    $\binom{a+b-2}{b-1}$ is even or $\binom{a}{b-1}$ is even.  That is, $B_0$ or $B_{b-2}$, respectively, fails to have the weak Lefschetz property in characteristic two.

    On the other hand, if $a = 2^m \ell$ and $b = 2^m+1$, where $m \geq 0$ and $l \geq 3$ odd, then for $0 < k < b-2$, $B_k$ fails to have the weak Lefschetz 
    property in characteristic two, by Proposition~\ref{lem:n-1-special-wlp}.  Note that $b \geq 4$ implies $b-2 \geq 2$.
\end{proof}

\subsection{Generation in a single degree}

Using the syzygy gap method in Subsection~\ref{sub:syzygy}, we get the following classification of the strong Lefschetz property for $R/I_{(d,d)}$.

\begin{theorem} \label{thm:slp-dd}
    Let $R = K[x,y]$, where $p$ is the characteristic of $K$, and $I_d = (x^d, y^d)$, where $d \geq 2$.  Then $R/I_d$ has the strong Lefschetz
    property if and only if $p = 0$ or $2d-2 < p^s$, where $s$ is the largest integer such that $p^{s-1}$ divides $(2d-1)(2d+1)$.
\end{theorem}
\begin{proof}
    By Theorem~\ref{thm:slp}, if $d \leq p \leq 2d-2$, then the strong Lefschetz property fails and if $p > 2d-2$ then the strong Lefschetz
    property holds.  By Corollary~\ref{cor:char-2-n-1}, if $p = 2$, then the strong Lefschetz property fails.  Hence, we need only to consider $2 < p < d$. 

    Next, notice that if such a triple $(u,v,w) \in \ZZ^3_{\rm odd}$ exists, then $u = v$ and so $w$ is odd.  Otherwise, if $u \neq v$, then
    $\left| m - u \right| + \left| m - v \right| \geq |u - v| \geq 1$ for all $m \in \RR$ by the triangle inequality; in particular, this 
    holds for $m = \frac{d}{p^s}$.

    Set $s$ to be the largest integer such that $p^{s-1}$ divides $(2d-1)(2d+1)$.  Further, set $e = 2d+1$, if $p$ divides $2d+1$, otherwise
    set $e = 2d-1$.

    As the sum $d + d + 2(d-1-k)$ is even, if $r = 0$, then $p^r(d,d,2(d-1-k))$ is at least one from every point in $\ZZ^3_{\rm odd}$, under the Manhattan distance.
    Suppose $0 < r < s$, then $p$ divides $e$; set $e = p^rn$ for some odd integer $n$ (recall $e$ is odd).  If $e = 2d-1$, then 
    $d = \frac{p^rn+1}{2}$.  The minimal value of $\left| \frac{d}{p^r} - u \right|$ is $\frac{p^r-1}{2p^r}$ at $u = \frac{n+1}{2}$. 
    The minimal value of $\left| \frac{2(d-1-k)}{p^r} - w \right|$ is $\frac{1+2k}{p^r}$ at $w = n$.  However, 
    $2\frac{p^r-1}{2p^r}+\frac{1+2k}{p^r} = \frac{p^r+2k}{p^r}$ is at least one for all $k$.  Similarly, if $e = 2d+1$, then the Manhattan
    distance to any point in $\ZZ^3_{\rm odd}$ is at least one.

    Suppose $2d-2 < p^s$.  Let $r \geq s$; then $2d-2 < p^r$, and so $\frac{d}{p^r} \leq \frac{1}{2}$.  Hence, we may set $u = v = 0$, 
    and thus $w \geq 1$ as $w$ must be odd.  As $2(d-1-k) \leq 2d-2 < p^r$, we must choose $w = 1$.  However, for all $k \geq 0$,
    \[
        \mu\left(\frac{1}{p^r}(d,d,2(d-1-k)), (0,0,1)\right) = \frac{p^r+2+2k}{p^r} > 1.
    \]

    Suppose $2d-2 \geq p^s$; then $e > p^s$, and we can write $e = p^sn+j$, where $p$ does not divide $n$ and $0 < j < p^s$
    ($j > 0$ as $p^s$ does not divide $e$).  Notice that $n > 0$ as $e > p^s$.  We consider two cases, given by the parity of $n$.

    Suppose $n$ is even, then $j$ is odd as $e$ is odd.  As $p^s$ is odd and $j$ is odd, then $j \neq p^s-1$ and $j \neq p^s-3$.
    Assume $e = 2d-1$, that is, $p$ does not divide $2d+1$.  Notice, $j \neq p^s-2$, otherwise, $2d-1 = p^s(n+1)-2$, and so 
    $2d+1 = p^s(n+1)$, which contradicts our choice of $e$.  Thus, $j \leq p^s-4$.  Set $u = v = \frac{n}{2}$, $w = n-1$, and $k = j+1$.
    As $n \geq 2$, $2p^s<e$ and so $p^s \leq d$.  This in turn implies $k = j + 1 < p^s-2 \leq d-2$; thus, $k$ is applicable. 
    As $e = 2d-1$, then $d = \frac{p^sn+j+1}{2}$.  Further,
    \begin{equation*}
        \begin{split}
            \mu\left( \frac{1}{p^s}(d,d,2(d-1-k)), (u,v,w) \right) &= 2\left| \frac{n}{2} + \frac{j+1}{2p^s} - \frac{n}{2} \right| + \left| n + \frac{-j-3}{p^s}-(n-1) \right| \\
                    &= \frac{j+1}{p^s} + \frac{p^s-j-3}{p^s} \\
                    &= \frac{p^s-2}{p^s} \\
                    &< 1.
        \end{split}
    \end{equation*}
    Assume $e = 2d+1$, that is, $p$ does divide $2d+1$.  In this case, set $u = v = \frac{n}{2}$, $w = n-1$, and $k = j$.
    Notice, $k \leq p^s-2$.  As $n \geq 2$, $2p^s<e$ and so $p^s \leq d$.  This in turn implies $k = j \leq p^s -2 \leq d-2$; thus, $k$ is applicable.
    As $e = 2d+1$, then $d = \frac{p^sn+j-1}{2}$.  Further,
    \begin{equation*}
        \begin{split}
            \mu\left( \frac{1}{p^s}(d,d,2(d-1-k)), (u,v,w) \right) &= 2\left| \frac{n}{2} + \frac{j-1}{2p^s} - \frac{n}{2} \right| + \left| n + \frac{-j-3}{p^s}-(n-1) \right| \\
                    &= \frac{j-1}{p^s} + \left|\frac{p^s-j-3}{p^s}\right| \hspace*{9em} (\star)\\
                    &= \frac{p^s-2}{p^s} \\
                    &< 1.
        \end{split}
    \end{equation*}
    Note for $(\star)$:  If $j \leq p^s-3$, then the absolute value disappears; on the other hand, if $j=p^s-2$, then the latter term is $\frac{1}{p^s}$.

    Suppose $n$ is odd, then $j$ is even as $e$ is odd.  Set $u = v = \frac{n+1}{2}$, $w = n$, and
    $k = \frac{j}{2}-1$.  Notice that $n \geq 1$ and $j \geq 2$.  As $j < p^s < 2d - 1 \leq e$, then $j \leq 2d-3$.  So 
    $k = \frac{j}{2}-1 \leq d-2$; thus, $k$ is applicable.  Suppose $e = 2d-1$, then $d = \frac{p^sn+j+1}{2}$.
    \begin{equation*}
        \begin{split}
            \mu\left( \frac{1}{p^s}(d,d,2(d-1-k)), (u,v,w) \right) &= 2\left| \frac{n}{2} + \frac{j+1}{2p^s} - \frac{n+1}{2} \right| + \left| n + \frac{1}{p^s}-n \right| \\
                    &= \frac{p^s-j-1}{p^s} + \frac{1}{p^s} \\
                    &= \frac{p^s-j}{p^s} \\
                    &< 1.
        \end{split}
    \end{equation*}
    When $e = 2d+1$, the result follows similarly with the finally fraction being $\frac{p^s-j+2}{p^s}$.  We notice that if $j = 2$, then
    $e = 2d+1 = p^sn-2$ and so $2d-1 = p^sn$.  That is, $p$ divides $2d-1$, contradicting our choice of $e$.  Thus, $j \geq 4$.
\end{proof}

\section{The presence of the Lefschetz properties for three variables}\label{sec:three}

In this section, we focus entirely on the strong Lefschetz property for $R/I_{(d,d,d)}$, where $d \geq 2$.  We
use the method of Kustin and Vraciu~\cite{KV} that is based on finding syzygies of low enough degree which we
recall next.

\subsection{Minimal degree syzygies} \label{sub:mgd}

Let $S = K[x_1, \ldots, x_n]$ and $\tup{d} = (d_0, d_1, \ldots, d_n) \in \NN^{n+1}$.  Define $\phi_{\tup{d}}: \oplus_{i=0}^{n}S(-d_i) \rightarrow S$
by the matrix $[(x_1+\cdots+x_n)^{d_0}, x_1^{d_1}, \ldots, x_n^{d_n}]$, and let $\syz(\tup{d}) := \ker{\phi_{\tup{d}}}$.  Next, define 
$\Kos(\tup{d})$ to be the $S$-submodule of $\syz(\tup{d})$ generated by the Koszul relations on the entries of the matrix defining $\phi_{\tup{d}}$,
and define $\overline{\syz}(\tup{d})$ to be the quotient $\syz(\tup{d}) / \Kos(\tup{d})$.  Last, for a non-zero graded module $M$, the 
{\em minimal generator degree} of $M$ is the smallest $d$ such that $M_d$ is non-zero; we denote this by $\mgd{M}$.

\begin{proposition}{\cite[Corollary~2.2(4 \& 6)]{KV}} \label{pro:mgd-wlp}
    Let $\tup{d} = (d_0, d_1, \ldots, d_n) \in \NN^{n+1}$, and set $t = d_0 + \cdots + d_n - (n+1)$.  Then $R/I_{\tup{d}}$ has the weak
    Lefschetz property if and only if $\flfr{t+3}{2} \leq \mgd{\overline{\syz}(\tup{d})}$.
\end{proposition}

Thus, $R/I_{\tup{d}}$ fails to have the weak Lefschetz property if we can demonstrate that there exists a non-Koszul syzygy of small enough degree.

\subsection{Finding syzygies}

First, we describe an explicit non-Koszul syzygy of $S/I_{(k,k+j,k+j,k)}$ that will be used repeatedly in the proceeding proof.
This is a generalisation of the syzygy described in the proof of~\cite[Lemma~4.2]{KV}.

\begin{lemma} \label{lem:std-syz}
    Let $j \in \NN_0$ and $k \in \NN$.  Then $(-f_{k+j}, g_k, (-1)^{k+j+1} g_k, f_{k+j})$ is a non-Koszul syzygy in $\overline{\syz}{(k,k+j,k+j,k)}$, where
    \[
        f_k := \frac{y^k-(-z)^k}{y+z} = \sum_{i=0}^{k-1} y^i(-z)^{k-i-1}
    \]
    and
    \[
        g_k := \frac{x^k-(x+y+z)^k}{y+z} = -\sum_{i=0}^{k-1} \binom{k}{i}x^i(y+z)^{k-i-1}.
    \]
\end{lemma}
\begin{proof}
    Notice that $(-f_{k+j}, g_k, (-1)^{k+j+1} g_k, f_{k+j}) \in \syz{(k,k+j,k+j,k)}$ as
    \begin{equation*}
        \begin{split}
             & \; -f_{k+j} x^k + g_k y^{k+j} + (-1)^{k+j+1} g_k z^{k+j} + f_{k+j} (x+y+z)^k \\
            =& \; f_{k+j} ((x+y+z)^k - x^k) + g_k (y^{k+j} - (-z)^{k+j}) \\
            =& \; \frac{y^{k+j} - (-z)^{k+j}}{y+z} ((x+y+z)^k - x^k) + \frac{x^k-(x+y+z)^k}{y+z} (y^{k+j} - (-z)^{k+j}) \\
            =& \; 0.
        \end{split}
    \end{equation*}

    Furthermore, it is clear that $f_{k+j} \not \in (x^k, y^{k+j}, z^{k+j})$ since $f_{k+j}$ is a polynomial in $y$ and $z$
    of degree $k+j-1$.  Thus, the described syzygy is non-Koszul.
\end{proof}

In order to demonstrate that the algebra $R/I_{(d,d,d)}$ does not have the strong Lefschetz property, we classify the
weak Lefschetz property for $S/I_{(d,d,d,d-3)}$.  

\begin{proposition} \label{pro:dddd-3}
    Let $d \geq 6$, and set $\tup{d} = (d,d,d,d-3)$.  Then $R/I_{\tup{d}}$ has the weak Lefschetz property in characteristic
    $p$ if and only if $p = 0$ or $p > 2d-3$.
\end{proposition}
\begin{proof}
    Set $\beta$ to be the quadruple $(x^d, y^d, z^d, (x+y+z)^{d-3})$.  The proof follows from several cases.

    {\em (i) Characteristic two}: Let $p = 2$.  If $d \neq 2^m+1$ for some $m \in \NN$, then $d = 2^m-k$ for some $0 \leq k \leq 2^{m-1}-2$,
    and $2d-3 = 2^{m+1}-2k-3 \geq 2^m+1$.  Thus, $d \leq 2^m \leq 2d-3$ and so $R/I_{\tup{d}}$ fails to have the weak Lefschetz
    property by Proposition~\ref{pro:wlp-naive}.  
    
    Suppose $d = 2^m+1$ for some $m \in \NN$.  Then $\alpha = (yz, xz, xy, xyz(x+y+z)^2)$
    is a syzygy in $\syz(\tup{d})$ as 
    \begin{equation*}
        \begin{split}
            \alpha \cdot \beta 
                &= x^{2^m+1}yz + xy^{2^m+1}z + xyz^{2^m+1} + xyz(x+y+z)^{2^m} \\
                &= xyz(x^{2^m} + y^{2^m} + z^{2^m} + (x+y+z)^{2^m}) \\
                &= xyz\left(x+y+z+(x+y+z)\right)^{2^m} \\
                &= 0.
        \end{split}
    \end{equation*}
    Further, $xyz(x+y+z)^2 \not \in (x^d, y^d, z^d)$ and $\deg{\alpha} = d + 2 \leq 2d-3$, as $d \geq 6$.  Hence by 
    Proposition~\ref{pro:mgd-wlp}, $R/I_{\tup{d}}$ fails to have the weak Lefschetz property. 

    Note that many of the following cases are proven almost identically to the case in the preceding paragraph.  
    In each of the forward cases, we provide only the syzygy, as the rest is straightforward to check.

    {\em (ii) Characteristic three}: Let $p = 3$, and write $2d = 3q + r$ with unique $q, r \in NN$ such that $0 \leq r \leq 2$.
    Suppose $q = 3^m$ and $r = 1$, then $d = 3^{m} + \frac{3^m+1}{2} = 3j - 1$, where $j = \frac{3^m+1}{2}$.  Let $\alpha$ be
    \[
        \left(x^{j-1}y^j(x+y+z)^{j-3}, (x-z)^{3^m}(x+y+z)^{j-3}, -y^{j}z^{j-1}(x+y+z)^{j-3}, -(x-z)^{3^m}y^{j}\right).
    \]
    Then $\alpha$ is in $\overline{\syz}(\tup{d})$, and $\deg{\alpha} = 2d-3$.  Thus, by Proposition~\ref{pro:mgd-wlp}, 
    $R/I_{\tup{d}}$ fails to have the weak Lefschetz property.

    Suppose $3^m < q \leq 2 \cdot 3^m-1$ for some $m$, then $d \leq \frac{6 \cdot 3^m - 3 + r}{2} \leq 3^{m+1}$ and 
    $2d-3 = 3q + (r-3) \geq 3(3^m+1) - 3 = 3^{m+1}$.  Thus, $R/I_{\tup{d}}$ fails to have the weak Lefschetz property by
    Proposition~\ref{pro:wlp-naive}.  
    
    Suppose $2\cdot 3^m \leq q < 3^{m+1}$ for some $m$.  Set $k = d-3^{m+1}$, so $0 \leq k \leq \frac{3^{m+1}-1}{2}$.
    Let $\alpha$ be
    \[  
        \left(y^k z^k (x+y+z)^j, x^k z^k (x+y+z)^j, x^k y^k (x+y+z)^j, -x^k y^k z^k (x+y+z)^{\max\{0, 3-k\}}\right),
    \]
    where $j = {\max\{0,k-3\}}$.  Then $\alpha$ is $\overline{\syz}(\tup{d})$, and $\deg{\alpha} \leq 2d-3$.  
    Thus, by Proposition~\ref{pro:mgd-wlp}, $R/I_{\tup{d}}$ fails to have the weak Lefschetz property.

    {\em (iii) Characteristic at least five}: Let $p \geq 5$ be prime, and let $f_k$ and $g_k$ be defined as in Lemma~\ref{lem:std-syz}.
    Write $2d = qp + r$ with unique $q, r \in NN$ such that $0 \leq r < p$.  Notice that $q$ and $r$ must have the same parity as $p$ is odd.
    We distinguish two sub-cases based on the parity of $q$ and $r$.

    {\em (a) The quotient is even}: Suppose $q$ and $r$ are even.  Set $j = {\max\{0,\frac{r}{2}-3\}}$, and $\alpha$ to be
    \[
        \left(
        y^{\frac{r}{2}} z^{\frac{r}{2}} (x+y+z)^j (-f_{\frac{q}{2}}^p),
        x^{\frac{r}{2}} z^{\frac{r}{2}} (x+y+z)^j g_{\frac{q}{2}}^p,
        \right.
    \]
    \[
        \left.
        x^{\frac{r}{2}} y^{\frac{r}{2}} (x+y+z)^j ((-1)^{\frac{q}{2}+1} g_{\frac{q}{2}})^p,
        x^{\frac{r}{2}} y^{\frac{r}{2}} z^{\frac{r}{2}} (x+y+z)^{\max\{0,3-\frac{r}{2}\}} f_{\frac{q}{2}}^p
        \right).
    \]
    Then $\alpha$ is in $\overline{\syz}(\tup{d})$, and $\deg{\alpha} \leq 2d-3$.  Thus, by Proposition~\ref{pro:mgd-wlp}, 
    $R/I_{\tup{d}}$ fails to have the weak Lefschetz property.

    {\em (b) The quotient is odd}:  Suppose $q$ and $r$ are odd.  First, suppose $r = 1$.  Then set $j = d-\frac{q-1}{2}p$, and $\alpha$ to be
    \[
        \left(
            (x+y+z)^{j-3} (-f_{\frac{q+1}{2}})^p,
            x^{j} y^{j-1} (x+y+z)^{j-3} g_{\frac{q-1}{2}}^p,
        \right.
    \]
    \[
        \left.
            x^{j} z^{j-1} (x+y+z)^{j-3} ((-1)^{\frac{q-1}{2}+2} g_{\frac{q-1}{2}})^p,
            x^{j} f_{\frac{q+1}{2}}^p
        \right).
    \]
    Notice that $d+j-1 = \frac{q+1}{2}p$.  Then $\alpha$ is in $\overline{\syz}(\tup{d})$, and $\deg{\alpha} = 2d-3$.  Thus, by 
    Proposition~\ref{pro:mgd-wlp}, $R/I_{\tup{d}}$ fails to have the weak Lefschetz property.

    Last, suppose $r\geq 3$.  Then set $j = d-r-\frac{q-1}{2}p$, and $\alpha$ to be
    \[
        \left(
            x^{j} (-f_{\frac{q+1}{2}})^p,
            y^{j} g_{\frac{q+1}{2}}^p,
            z^{j} ((-1)^{\frac{q+1}{2}+1} g_{\frac{q+1}{2}})^p,
            (x+y+z)^{j+3} f_{\frac{q+1}{2}}^p
        \right).
    \]
    Then $\alpha$ is in $\overline{\syz}(\tup{d})$, and $\deg{\alpha} \leq 2d-3$.  Thus, by Proposition~\ref{pro:mgd-wlp}, 
    $R/I_{\tup{d}}$ fails to have the weak Lefschetz property.
\end{proof}

\begin{remark}
    Each of the syzygies described in the preceding proof are modifications of extant syzygies by means of the Frobenius homomorphism and 
    multiplying by an appropriate ring element.  This is similar to the approach used by Kustin and Vraciu~\cite{KV}.

    Further, to discuss the cases left out in Proposition~\ref{pro:dddd-3}, we notice that the determinants associated to $(4,4,4,1)$,
    and $(5,5,5,2)$ are $20 = 2^2 \cdot 5$ and $-43750 = -2 \cdot 5^5 \cdot 7$, respectively.  Thus, $R/I_{(4,4,4,1)}$ fails
    to have the weak Lefschetz property in exactly characteristics $2$ and $5$.  Similarly, $R/I_{(5,5,5,2)}$ fails to have
    the weak Lefschetz property in exactly characteristics $2, 5,$ and $7$.
\end{remark}

\begin{theorem} \label{thm:slp-ddd}
    Let $d \geq 2$, and set $\tup{d} = (d,d,d)$.  Then $R/I_{\tup{d}}$ has the strong Lefschetz property in characteristic $p$ if
    and only if $p = 0$ or $p > 3(d-1)$.
\end{theorem}
\begin{proof}
    By Theorem~\ref{thm:slp}, if $d \leq p \leq 3(d-1)$, then $R/I_{\tup{d}}$ fails to have the strong Lefschetz property, and if
    $p > 3(d-1)$, then $R/I_{\tup{d}}$ has the strong Lefschetz property.

    If $d \geq 6$, then by Proposition~\ref{pro:dddd-3}, for $2 \leq p < d$, $S/I_{(d,d,d,d-3)}$ fails to have the weak Lefschetz
    property.  Thus by Proposition~\ref{pro:slp}, $R/I_{\tup{d}}$ fails to have the strong Lefschetz property for $2 \leq p < d$
    as $d-3 = t - 2k$, where $t = 3d-3$ and $k = d$.

    For the remaining four cases, we consider $k = 0$ and use Lemma~\ref{lem:large-top}.  In particular, notice that
    $\binom{3}{1,1,1} = 2 \cdot 3$, $\binom{6}{2,2,2} = 2 \cdot 3^2 \cdot 5$, $\binom{9}{3,3,3} = 2^4 \cdot 3 \cdot 5 \cdot 7$,
    and $\binom{12}{4,4,4} = 2 \cdot 3^2 \cdot 5^2 \cdot 7 \cdot 11$.  Hence, for $2 \leq d \leq 5$, $S/I_{(d,d,d,t)}$ fails
    to have the weak Lefschetz property for $2 \leq p < d$, and so $R/I_{(d,d,d)}$ fails to have the strong Lefschetz property.
\end{proof}

\section{The presence of the Lefschetz properties in many variables}\label{sec:many}

We first consider the strong Lefschetz property in characteristic two when $n \geq 2$, that is, when $R$ has at least three variables.
Then we consider the strong Lefschetz property for $I_{\tup{d}}$ having generators of the same degree $d_0 = \cdots = d_n$ in at least
four variables.

\subsection{Characteristic two}

We expand Corollary~\ref{cor:Kummer} to classify when multinomial coefficients are odd.

\begin{lemma} \label{lem:odd-multi}
    Let $a_0 \geq \cdots \geq a_n \geq 1$.  Then the following are equivalent:
    \begin{enumerate}
        \item $\binom{a_0 + \cdots + a_n}{a_0, \ldots, a_n}$ is odd,
        \item $\BP(a_i)$ and $\BP(a_j)$ are disjoint for all $0 \leq i < j \leq n$, and
        \item $\BP(a_{i_1} + \cdots + a_{i_m})$ and $\BP(a_j)$ are disjoint for any $1 \leq m < n$ and $j \not\in \{i_1, \ldots, i_m\} \subsetneq [n]$.
    \end{enumerate}
\end{lemma}
\begin{proof}
    Set $M = \binom{a_0 + \cdots + a_n}{a_0, \ldots, a_n}$.

    {\em (i) $\Rightarrow$ (ii)}:  Notice $\binom{a_i+a_j}{a_i}$ divides $M$ for all $0 \leq i < j \leq n$.  Thus, if $M$ is odd, then so
    is $\binom{a_i+a_j}{a_i}$.  Hence, by Corollary~\ref{cor:Kummer}, $\BP(a_i)$ and $\BP(a_j)$ are disjoint.

    {\em (ii) $\Rightarrow$ (iii)}: Let $1 \leq m < n$ and $j \not\in \{i_1, \ldots, i_m\} \subsetneq [n]$.  As $\BP(a_0), \ldots, \BP(a_n)$
    are disjoint, $\BP(a_{i_1} + \cdots + a_{i_m}) = \dotcup B_{i_k}$ is disjoint from $\BP(a_j)$.

    {\em (iii) $\Rightarrow$ (i)}: Recall that 
    \[
        M = \binom{a_0 + \cdots + a_n}{a_0, \ldots, a_n} = \prod_{i=2}^n \binom{a_1 + \cdots + a_i}{a_i}.
    \]
    As $\BP(a_1 + \cdots + a_{i-1})$ and $\BP(a_i)$ are disjoint, by Corollary~\ref{cor:Kummer}, $\binom{a_1 + \cdots + a_i}{a_i}$ is odd.  Hence
    $M$ is a product of odd integers, that is, $M$ is odd.
\end{proof}

By the preceding lemma, for certain pairs of multinomial coefficients, one must be even.

\begin{lemma} \label{lem:one-or-the-other}
    Let $n \geq 2$, $a_0 \geq \cdots \geq a_n \geq 1$, and suppose $a_0 \geq a_1 + \cdots + a_n$.  Then
    $\binom{a_0 + \cdots + a_n}{a_0, \ldots, a_n}$ is even or $\binom{a_0 + 1}{a_0 + 1 - (a_1 + \cdots + a_n), a_1, \ldots, a_n}$ is even.
\end{lemma}
\begin{proof}
    Suppose $\binom{a_0 + \cdots + a_n}{a_0, \ldots, a_n}$ and $\binom{a_0 + 1}{a_0 + 1 - (a_1 + \cdots + a_n), a_1, \ldots, a_n}$ are odd.
    By Lemma~\ref{lem:odd-multi}, the $\BP(a_i)$ are disjoint for all $0 \leq i \leq n$, $\BP(a_0)$ and $\BP(a_1 + \cdots + a_n)$ are disjoint,
    and $\BP(a_0 + 1 - (a_1 + \cdots + a_n))$ and $\BP(a_1 + \cdots + a_n)$ are disjoint.  Thus we have that 
    $\BP(a_0 + 1) = \BP(a_0 + 1 - (a_1 + \cdots + a_n)) \dotcup \BP(a_1 + \cdots + a_n)$.    Notice that since $a_n \geq 1$, each $\BP(a_i)$ has at 
    least one element, and so $\BP(a_1 + \cdots + a_n)$ has at least $n$ elements.

    Suppose $\BP(a_0)$ contains $0, \ldots, m-1$ but not $m$.  Then for $k > m$, $k \in \BP(a_0 + 1)$ if and only if $k \in \BP(a_0)$.
    Moreover, $\BP(a_0 + 1)$ contains $m$ but not $0, \ldots, m-1$.  As $\BP(a_1 + \cdots + a_n) \subset \BP(a_0+1)$, and the former has at least
    $n \geq 2$ elements, there exists a $k \in \BP(a_1 + \cdots + a_n) \subset \BP(a_0+1)$ with $k > m$.  Thus $\BP(a_1 + \cdots + a_n)$ and $\BP(a_0)$
    have $k$ in common, contradicting $\BP(a_0)$ and $\BP(a_1 + \cdots + a_n)$ being disjoint.  This in turn contradicts $\binom{a_0 + \cdots + a_n}{a_0, \ldots, a_n}$
    being odd.
\end{proof}

As a corollary, we classify the strong Lefschetz property in characteristic two for all monomial complete intersections in at least three variables.

\begin{corollary} \label{cor:char-2-n-2}
    Let $d_0 \geq \cdots \geq d_n \geq 2$ with $n \geq 2$.  Then $R/I_{\tup{d}}$ fails to have the strong Lefschetz property in characteristic two.
\end{corollary}
\begin{proof}
    If $d_0 \leq \clfr{t}{2}$, then $\frac{t+1}{d_0} \geq 2$ and so $d_0 \leq 2^m \leq t$, for some $m \in \NN$.  Thus, by Theorem~\ref{thm:slp}
    $R/I_{\tup{d}}$ fails to have the strong Lefschetz property in characteristic two.

    Set $\ell := x_0 + \cdots + x_n$ and $B_k := R[x_{n+1}]/I_{(\tup{d}, t-2k)}$.  Recall that by Proposition~\ref{pro:slp}, $R/I_{\tup{d}}$ has the
    strong Lefschetz property if and only if each $B_k$ has the weak Lefschetz property, for $0 \leq k \leq \flfr{t}{2}$. 

    Suppose $d_0 > \clfr{t}{2}$, that is, $d_0 \geq d_0 + \cdots + d_n - n$.  Notice $B_0$ has the weak Lefschetz property in characteristic $2$
    if and only if $\binom{t}{d_0 - 1, \ldots, d_n - 1}$ is odd.  Further, $t-d_0 \leq \flfr{t}{2}$ as $d_0 > \clfr{t}{2}$, and $B_{t-d_0}$ has
    the weak Lefschetz property in characteristic $2$ if and only if $\binom{d_0}{d_1 - 1, \ldots d_n - 1, d_0 - (d_1 + \cdots + d_n - n)}$ is odd 
    (notice that $t - 2(t-d_0) = 2d_0 - t = d_0 + 1 - (d_1 + \cdots + d_n - n)$).

    By Lemma~\ref{lem:one-or-the-other}, $\binom{t}{d_0 - 1, \ldots, d_n - 1}$ is even or $\binom{d_0}{d_1 - 1, \ldots d_n - 1, d_0 - (d_1 + \cdots + d_n - n)}$
    is even, thus $B_0$ or $B_{t-d_0}$ fails to have the weak Lefschetz property in characteristic $2$.  Hence, $R/I_{\tup{d}}$ fails to have the 
    strong Lefschetz property in characteristic two.
\end{proof}

\subsection{Generation in a single degree}

In this subsection, we consider monomial complete intersections generated by monomials of the same degree, that is, $d_0 = \cdots = d_n = d \geq 2$.  Notice
that the socle degree is $(n+1)(d-1)$.

The case when $n = 1$ is handled in Section~\ref{sec:two}.  In particular, the weak Lefschetz property is classified in Proposition~\ref{pro:2-wlp},
and the strong Lefschetz property is classified in Theorem~\ref{thm:slp-dd}.  Brenner and Kaid~\cite[Theorem~2.6]{BK-p} classify the weak Lefschetz
property when $n = 2$.  We note that Kustin, Rahmati, and Vraciu~\cite{KRV} relate this result to the projective dimension of 
$K[x,y,z]/(x^d, y^d, z^d) \colon (x^n+y^n+z^n)$.  Kustin and Vraciu~\cite[Theorem~4.3]{KV} classify the weak Lefschetz property when $n = 3$.
Further still, Kustin and Vraciu~\cite{KV} prove the surprising classification of the weak Lefschetz property when $n \geq 4$.  We recall the last here,
as we will use it.

\begin{theorem}{\cite[Theorem~6.4]{KV}} \label{thm:d-fixed-many-var}
    Let $d \geq 2$ and $n \geq 4$.  Then $R/I_{(d, \ldots, d)}$ has the weak Lefschetz property if and only if the
    characteristic of $K$ is $0$ or greater than $\clfr{(n+1)(d-1)}{2}$.
\end{theorem}

As a corollary of the above theorem, we get a classification of the strong Lefschetz property when $n \geq 4$.

\begin{corollary} \label{cor:d-fixed-slp}
    Let $d \geq 2$ and $n \geq 4$.  Then $R/I_{(d, \ldots, d)}$ has the strong Lefschetz property if and only if the
    characteristic of $K$ is $0$ or greater than $(n+1)(d-1)$.
\end{corollary}
\begin{proof}   
    As the strong Lefschetz property implies the weak Lefschetz property, we combine Theorems
    \ref{thm:slp} and \ref{thm:d-fixed-many-var} to verify the claim.
\end{proof}

Before classifying the strong Lefschetz property for $n = 3$, we prove a more general lemma regarding the 
weak Lefschetz property and monomial complete intersections generated by monomials all having the same degree, except one.

\begin{lemma} \label{lem:d-odd-fixed}
    Let $\tup{d} \in \NN^{n+1}$, where $d \geq 3$, $n \geq 4$, $d_0 = \cdots = d_{n-1} = d$ and $d_n = d-1$.  If $d$ or $n$ is odd, and 
    $2 \leq p < d$, then $R/I_{\tup{d}}$ fails to have the weak Lefschetz property.
\end{lemma}
\begin{proof}
    Set $\tup{e} = (d, \ldots, d) \in \NN^{n+1}$.

    By Theorem~\ref{thm:d-fixed-many-var}, $R/I_{\tup{e}}$ fails the weak Lefschetz property for $2 \leq p < d$.
    Thus, by Proposition~\ref{pro:mgd-wlp}, $\mgd \overline{\syz} \tup{d} < \flfr{(n+1)(d-1)+3}{2}$.  Notice that 
    $\flfr{(n+1)(d-1)+3}{2} = \frac{(n+1)(d-1)+2}{2}$ as $(n+1)(d-1)+3$ is odd.

    Let $\alpha = (z_0, \ldots, z_n)$ be a homogeneous representative of a nonzero syzygy in $\overline{\syz} \tup{e}$
    of degree $\mgd \overline{\syz} \tup{e}$ such that $z_0 \not\in (x_1^d, \ldots, x_n^d)$.  Without loss of generality
    we may further assume $x_n^{d-1}$ does not divide $z_0$ (otherwise, the degree of $\alpha$ would be at least $(n+1)(d-1)$,
    which is larger than $\frac{(n+1)(d-1)+2}{2}$), that is, $z_0 \not\in (x_1^d, \ldots, x_{n-1}^d, x_n^{d-1})$.

    Then $\alpha' = (z_0, \ldots, x_nz_n)$ is a homogeneous nonzero syzygy in $\syz \tup{d}$.  Further, as $z_0$ is
    not a member of $(x_1^d, \ldots, x_{n-1}^d, x_n^{d-1})$, and all relations in $\Kos \tup{d}$ must have a $S$-linear combination of
    $x_1^d, \ldots, x_{n-1}^d, x_n^{d-1}$ in the first entries, then $\alpha'$ is not in $\Kos \tup{d}$.  Thus, $\alpha'$ is
    a homogeneous representative of a nonzero syzygy in $\overline{\syz} \tup{d}$.  

    Notice, the degree of $\alpha'$ is the degree of $\alpha$, and is strictly bounded above by $\frac{(n+1)(d-1)+2}{2}$.  Hence,
    $\mgd \overline{\syz} \tup{d} < \frac{(n+1)(d-1)+2}{2}$.  Notice though, $\flfr{n(d-1) + (d-2) + 3}{2} = \frac{(n+1)(d-1)+2}{2}$.  
    Therefore, by~\ref{pro:mgd-wlp}, $R/I_{\tup{d}}$ fails to have the weak Lefschetz property for $2 \leq p < d$.
\end{proof}

From this we get a classification of the strong Lefschetz property when $n = 3$.

\begin{proposition} \label{pro:slp-dddd}
    Let $d \geq 2$.  Then $R/I_{(d, d, d, d)}$ has the strong Lefschetz property if and only if the
    characteristic of $K$ is $0$ or greater than $4(d-1)$.
\end{proposition}
\begin{proof}   
    By Theorem~\ref{thm:slp}, we need only to consider $2 \leq p < d$.

    By Proposition~\ref{pro:slp}, $R/I_{(d, d, d, d)}$ fails the strong Lefschetz property if $R[z]/I_{(d,d,d,d,4d-4-2k)}$
    fails to have the weak Lefschetz property for some $0 \leq k \leq 2d-2$.  
    
    Suppose $d$ is even, then $4d-4-2k = d$ when $k = \frac{3d-4}{2} < 2d-2$.  Thus, using Theorem~\ref{thm:d-fixed-many-var} 
    we see that $R/I_{(d, d, d, d)}$ fails to have the strong Lefschetz property for $p \leq \clfr{5(d-1)}{2}$.  
    As $d < \clfr{5(d-1)}{2}$ for all $d$, then the claim holds.

    Suppose $d$ is odd, then $4d-4-2k = d-1$ when $k = \frac{3d-3}{2} < 2d - 2$. Thus, using Lemma~\ref{lem:d-odd-fixed} we 
    see that $R/I_{(d, d, d, d)}$ fails to have the strong Lefschetz property for $2 \leq p < d$.  
\end{proof}

\section{Conclusions} \label{sec:conclusions}

We combine Corollaries~\ref{cor:char-2-n-1} and~\ref{cor:char-2-n-2} to get the following theorem classifying the 
strong Lefschetz property for monomial complete intersections in characteristic two.

\begin{theorem} \label{thm:classify-char-two}
    Let $d_0 \geq \cdots \geq d_n \geq 2$ with $n \geq 1$, and let $I = (x_0^{d_0}, \ldots, x_n^{d_n}) \subset R = K[x_0, \ldots, x_n]$,
    where $K$ is an infinite field of characteristic two.  Then $R/I$ has the strong Lefschetz property if and only if
    $n = 1$ and either (i) $d_0$ is odd and $d_1 = 2$ or (ii) $d_0 = 4k + 2$ for some $k \in \NN$ and $d_1 = 3$.
\end{theorem}

Moreover, combining Theorem~\ref{thm:slp-dd} ($n = 1$), Theorem~\ref{thm:slp-ddd} ($n = 2$), Proposition~\ref{pro:slp-dddd} ($n = 3$),
and Corollary~\ref{cor:d-fixed-slp} ($n \geq 4$), we completely classify the strong Lefschetz property for monomial complete intersections
generated by monomials all having the same degree.

\begin{theorem} \label{thm:classify-fixed-d}
    Let $d \geq 2$, $n \geq 1$, and $I = (x_0^d, \ldots, x_n^d) \subset R = K[x_0, \ldots, x_n]$, where $K$ is an infinite field of
    characteristic $p$.  Then $R/I$ has the strong Lefschetz property if and only if $p$ is zero or $p$ is a positive prime and either
    \begin{enumerate}
        \item $n = 1$ and $p^s > 2(d-1)$, where $s$ is the largest integer such that $p^{s-1}$ divides $(2d-1)(2d+1)$, or
        \item $n \geq 2$ and $p > (n+1)(d-1)$.
    \end{enumerate}
\end{theorem}

By Theorem~\ref{thm:slp}, for a monomial complete intersection generated in degrees $d_0 \geq \cdots \geq d_n \geq 2$, the presence
of the strong Lefschetz property is uniform for primes at least $d_0$.  However, for small primes (those less than $d_0$), the
strong Lefschetz property appears to behave chaotically when arbitrary degree sequences $\tup{d} = (d_0, \ldots, d_n)$ are considered.  
However, some restrictions, such as characteristic two or a fixed generating degree, can limit this apparent chaos to only the case of two 
variables.  This suggests that perhaps more focus should be given to two variables.

\begin{question}
    For which prime characteristics $p$ does the algebra $K[x,y]/(x^a, y^b)$, where $a \geq b \geq 2$, fail to have the strong Lefschetz property?
\end{question}

Unfortunately, Proposition~\ref{pro:wlp-naive} has a gap when the socle degree $t$ is even and $p = \frac{t}{2}+1$.  Experimentally, the weak Lefschetz property always holds
in this case.  However, Corollary~\ref{cor:wlp-even} cannot be used in this specific case.  As an example, consider $A = K[w,x,y]/(w^5, x^5, y^5)$
and $B = K[w,x,y,z]/(w^5, x^5, y^5, z^2)$.  In this case, $A$ has the weak Lefschetz property in characteristic $7$, but $B$ does not.

We formalise the above experimental results.

\begin{conjecture} \label{conj:wlp-even-gap}
    Let $t$ be the socle degree of $R/I_{\tup{d}}$.  If $t$ is even and the characteristic of $K$ is $p = \frac{t}{2}+1$, then
    the algebra $R/I_{\tup{d}}$ has the weak Lefschetz property.
\end{conjecture}

Conjecture~\ref{conj:wlp-even-gap} is true when $n = 2$.  

\begin{remark}
    Let $a \geq b \geq c \geq 2$ such that $t = a+b+c-3 = 2(p-1)$ for some prime $p$, and suppose that $a \leq \clfr{t}{2} = p - 1$.
    Set $A = K[x,y,z]/(x^a, y^b, z^c)$.  Let $\alpha = p - b$ and $\beta = b - 1$, and notice that $0 < \alpha < a$.  
    Consider the following commutative diagram, where $B = K[x,y,z]/(x^a, y^b, z^c, x^{\alpha} y^{\beta})$ and $\ell = x+y+z$.

    \begin{displaymath}
        \xymatrix{
            [A]_{p-2} \ar[r]^{\times\ell} \ar[d]^{\cong} & [A]_{p-1} \\
            [B]_{p-2} \ar[r]^{\times\ell} & [B]_{p-1} \ar@{^{(}->}[u] \\
        }
    \end{displaymath}
    
    Thus, the top map is injective if the bottom map is.  Using~\cite[Proposition~6.12]{CN-Enum}, we see that $B$ has the weak
    Lefschetz property, and thus the bottom map is injective, if the characteristic of $K$ is at least $\frac{a+b+c+\alpha+\beta}{3} = p$.
    Hence the top map is injective and $A$ has the weak Lefschetz property in characteristic $p$.
\end{remark}

Moreover, through experiments using {\em Macaulay2}~\cite{M2}, we conjecture that when $d_0$ is ``small'' (i.e., when the weak Lefschetz property is
{\em not} guaranteed to hold by Proposition~\ref{pro:large-degree}), then the strong Lefschetz property only holds when guaranteed by 
Theorem~\ref{thm:slp}.  Notice that Theorems~\ref{thm:classify-char-two} and~\ref{thm:classify-fixed-d} provide evidence for this conjecture.

\begin{conjecture} \label{conj:slp-small}
    Suppose $d_0 \leq \clfr{t}{2}$, where $t$ is the socle degree of $R/I_{\tup{d}}$.  Then $R/I_{\tup{d}}$ has the strong Lefschetz property
    if and only if the characteristic of $K$ is either $0$ or greater than $t$.
\end{conjecture}

\begin{acknowledgement}
    The author would like to acknowledge the invaluable nature of the computer algebra system {\em Macaulay2}~\cite{M2}.
    {\em Macaulay2} was used extensively throughout both the original research and the writing of this manuscript. 

    The author would also like to thank his advisor, Uwe Nagel, for many discussions about this manuscript.
\end{acknowledgement}


\end{document}